\documentclass[11pt,a4paper]{article}
\usepackage[OT1]{fontenc}
\usepackage[english]{babel}
\usepackage{amsfonts}
\usepackage{amsthm}
\def\h{ {\cal H} }
\def\a{ {\cal A} }

\def\l{ {\cal L} }
\def\n{ {\cal N} }
\def\v{ {\cal V} }

\def\b{ {\cal B} }
\def\u{ {\cal U} }

\def\t{ {\cal T} }
\def\s{ {\cal S} }

\def\p{ {\cal P} }

\def\k{ {\cal K} }
\def\f{ {\cal F} }
\def\c{ {\cal C} }

\def\j{ {\cal J} }

\newtheorem{teo}{Theorem}[section]
\newtheorem{prop}[teo]{Proposition}
\newtheorem{lem}[teo]{Lemma}
\newtheorem{coro}[teo]{Corollary}
\newtheorem{defi}[teo]{Definition}

\theoremstyle{definition}
\newtheorem{rem}[teo]{Remark}
\newtheorem{ejem}[teo]{Example}
\newtheorem{ejems}[teo]{Examples}

\title{Essentially orthogonal subspaces}
\author{Esteban Andruchow and Gustavo Corach}

\begin{document}

\maketitle 

\begin{abstract}
We study the set $\c$ consisting of pairs of orthogonal projections $P,Q$ acting in a Hilbert space $\h$  such that $PQ$ is a compact operator. These pairs have a rich geometric structure which we describe here. They are parted in three subclasses: $\c_0$ which consists of pairs where $P$ or $Q$ have finite rank, $\c_1$ of pairs such that $Q$ lies in the restricted Grassmannian (also called Sato Grassmannian) of the polarization $\h=N(P)\oplus R(P)$, and $\c_\infty$. Belonging to this last subclass one has the pairs
$$
P_If=\chi_If  ,\  \ Q_Jf= \left(\chi_J \hat{f}\right)\check{\ } , \ \ f\in L^2(\mathbb{R}^n),
$$ 
where $I, J\subset \mathbb{R}^n$ are sets of finite Lebesgue measure, $\chi_I, \chi_J$ denote the corresponding characteristic functions and $\hat{\ } , \check{\ }$ denote the Fourier-Plancherel transform $L^2(\mathbb{R}^2)\to L^2(\mathbb{R}^2)$ and its inverse. 
We characterize the connected components of these classes: the components of $\c_0$ are parametrized by the rank, the components of $\c_1$ are parametrized by the Fredholm index of the pairs, and $\c_\infty$ is connected. We show that these subsets are (non complemented) differentiable submanifolds of $\b(\h)\times \b(\h)$.
\end{abstract}

\bigskip

{\bf 2010 MSC:}  58B20, 47B15, 42A38, 47A63.

{\bf Keywords:}  Projections, pairs of projections, compact operators, Grasmann manifold.

\section{Introduction}
The study of pairs of subspaces of a Hilbert space  $\h$ or, more generally, pairs of orthogonal projections in a C$^*$-algebra started in the early times of spectral theory with Dixmier \cite{dixmier}. Some efforts towards finding more transparent proofs of Dixmier's theorems are due to Davis \cite{davis}, Pedersen \cite{pedersen}, Halmos \cite{halmos}, Raeburn and Sinclair \cite{raesin},  Avron, Seiler and Simon \cite{ass}, Amrein and Sinha \cite{amreinsinha}, among many others. The excellent survey of B\"ottcher and Spitkovsky \cite{gentle} contains a complete description and bibliography. This theory  is  concerned not only with two projections $P$, $Q$ in $\b(\h)$ (the algebra of bounded linear operators in $\h$) but also with the products $PQ$ and $PQP$. This paper is an addition to this part of the theory, where $PQ$ is supposed to be compact.  The interest in this type of products is not new. Consider the following examples.

\begin{ejems}\label{el ejemplo}

\begin{enumerate}
\item
Let $I,J\subset \mathbb{R}^n$ be Lebesgue-measurable sets of finite measure. Let $P_I,Q_J$ be the projections in $L^2(\mathbb{R}^n,dx)$ given by
$$
P_If=\chi_If \ \ \hbox{ and } \ \ Q_Jf= \left(\chi_J \hat{f}\right)\check{\ },
$$
where $\chi_L$ denotes the characteristic function of the set $L$.
Equivalently, denoting by $U_\f$ the Fourier transform regarded as a unitary operator acting in $L^2(\mathbb{R}^n,dx)$, then
$$
P_I=M_{\chi_I} \ \hbox{ and } Q_J=U_\f^*M_{\chi_J}U_\f .
$$
In  \cite{donoho} (Lemma 2) it is proven that $P_IQ_J$ is a Hilbert-Schmidt operator. See also \cite{folland}. These products play a relevant role in operator theoretic formulations of the uncertainty principle \cite{donoho}, \cite{folland}.
\item
Let $\h=L^2(\mathbb{T},dt)$  
where $\mathbb{T}$ is the unit circle, and consider the decomposition 
$$
\h=\h_-\oplus\h_+,
$$
where $\h_+$ is the Hardy space. Benote by $P_+$ and $P_-$ the orthogonal projections onto $\h_+$ and $\h_-$, respectively. Let  $\varphi, \psi$ be continuous  functions on $\mathbb{T}$ with $|\varphi(e^{it})|=|\psi(e^{it})|=1$ for all $t$, and
$$
P=P_{\varphi\h_+}^\perp=1-P_{\varphi\h_+} \ , \ \ Q=P_{\psi\h_+}.
$$
Since $\varphi$ and $\psi$ are unimodular, the multiplication operators $M_\varphi$, $M_\psi$ are unitary in $\h$. Then
$$
P_{\varphi\h_+}=P_{M_\varphi(\h_+)}=M_\varphi P_+ M_\varphi^*=M_\varphi P_+ M_{\bar{\varphi}},
$$
and similarly for $P_{\psi \h_+}$. 
Then 
$$
P_{\varphi\h_+}^\perp=1-P_{\varphi\h_+}=M_\varphi (1-P_+) M_{\bar{\varphi}}=M_\varphi P_- M_{\bar{\varphi}}.
$$
Therefore
$$
PQ=M_\varphi P_- M_{\bar{\varphi}}M_\psi P_+ M_{\bar{\psi}}=M_\varphi P_- M_{\bar{\varphi}\psi}P_+ M_{\bar{\psi}}.
$$
Since $P_- M_{\bar{\varphi}\psi}|_{\h_+}$ is a Hankel operator with continuous symbol, it follows  by Hartman's theorem \cite{hartman} that it is  compact (see also Theorem 5.5 in  \cite{peller}). Thus $PQ$ is compact. 
\end{enumerate}
We shall see below that these two examples are of  different nature.
\end{ejems}

  Our main goal in this paper is the study of the geometry of the sets
  $$
  \c=\{(P,Q): P,Q\hbox{ are orthogonal projections and } PQ \hbox{ is compact}\}
  $$
  and, for each projection $P$,
  $$
  \c(P)=\{Q: PQ \hbox{ is compact}\}.
  $$
Let us describe the contents of the paper.

In Section 2 we state elementary properties of pairs $P, Q$ in $\c$: the spectral description of the entries of $Q$,  written as a $2\times 2$ matrix in terms of $P$, and the partition of the class $\c$ in three subclasses $\c_0, \c_1$ and $\c_\infty$.  In Section 3 we recall the so-called  Halmos decomposition of $\h$ given by a pair of subspaces, and specialize it  to the case where the corresponding pair of projections lies in $\c$. In Section 4 we give a spatial characterization of $\c$ in the following sense: given orthogonal projections $P$, $Q$, denote $\s=R(P)$ and $\t=R(Q)$;  then $(P,Q)$ belongs to $\c$ if and only if there exist orthonormal bases $\{\xi_n\}$, $\{\psi_n\}$ of $\s$ and $\t$, respectively, such that $\langle\xi_n,\psi_k\rangle=0$ if $n\ne k$, and $\langle\xi_n,\psi_n\rangle\to 0$.  In Section 5 we introduce the action of the restricted unitary group induced by $P$ on projections $Q\in\c(P)$. In Section 6 we study the class $\c_1$, on which an index is defined, and prove that the connected components of $\c_1$ are parametrized by this index. It is shown that Hankel pairs as in Example 1.1.2 belong to $\c_1$.  In Section 7 we study the class $\c_\infty$, and prove that it is connected. We also prove that the pairs $(P_I,Q_J)$ like in Example 1.1.1 belong to $\c_\infty$. In Section 8 we prove that the sets $\c$ and $\c(P)$ are (non complemented) C$^\infty$-differentiable submanifolds of $\b(\h)$.
\section{Elementary properties}

Let $\h$ be a Hilbert space, $\b(\h)$ the algebra of bounded linear operators in $\h$, $\k(\h)$ the ideal of compact operators and $\p(\h)$ the set of selfadjoint (orthogonal) projections. If $\s$ is a closed subspace of $\h$, the orthogonal projection onto $\s$ is denoted by $P_\s$.
Given $P\in\p(\h)$, operators acting in $\h$ can be written as $2\times 2$ matrices. For instance, any projection $Q$  is of the form
$$
Q=\left( \begin{array}{cc} a & x \\ x^* & b \end{array} \right)
$$
where the fact that $Q$ is a projection is equivalent to the relations
\begin{equation}\label{relaciones proyeccion}
\left\{ \begin{array}{l} xx^*=a-a^2 \\ x^*x=b-b^2  \\ ax+xb=x \end{array} \right. ,
\end{equation}
with $0\le a\le 1_{R(P)}$, $0\le b\le 1_{N(P)}$ and $\|x\|\le 1/2$. The fact that $PQ$ is compact means that $a\in\b(R(P))$ and $x\in\b(N(P),R(P))$ are compact. Here and throughout $R(T)$ and $N(T)$ denote the range and the nullspace of $T$, respectively.

Let us show another example of pairs of projections with compact product.
\begin{ejem}\label{ejemploClaseE}
Let $\h=\l\times\s$ and fix a compact operator $K:\s\to \l$. Consider the idempotent $E=E_K$ given by the matrix
$$
E=\left( \begin{array}{cc}  1_\l & K \\ 0 & 0 \end{array} \right).
$$  
Then $P=P_{R(E)}=P_\l$ and $Q=P_{N(E)}$ satisfy that $PQ$ is compact. Indeed, straightforward computations show that $R(E)=\l$ and that
$$
P_{N(E)}=(1-E_K)(1-E_K-E_K^*)^{-1}=\left( \begin{array}{cc}  KK^*(1+KK^*)^{-1} & -K(1+K^*K)^{-1} \\ -K^*(1+KK^*)^{-1} & (1+K^*K)^{-1} \end{array} \right). 
$$
Then 
$$
PQ=\left( \begin{array}{cc}   KK^*(1+KK^*)^{-1} & -K(1+K^*K)^{-1} \\ 0 & 0 \end{array} \right),
$$
which is clearly compact. The singular values of $PQ$ are the square roots of the eigenvalues of 
$$
PQP=\left( \begin{array}{cc}   KK^*(1+KK^*)^{-1} & 0 \\ 0 & 0 \end{array} \right),
$$
i.e., those of $KK^*(1+KK^*)^{-1}$, which have the same asymptotic behaviour near zero as the singular values of $K$. 
\end{ejem}

Let us collect several elementary properties of pairs in $\c$. First note that $b$ (in the matrix expression of $Q$ in terms of $P$) may not be compact. It is positive and $b-b^2$ is compact. This implies that it can be diagonalized, and that its spectrum consists of eigenvalues which can only accumulate (eventually) at $0$ or $1$, plus $0$ and $1$ which may not be eigenvalues. All spectral values different from $0$ or $1$ have finite multiplicity.

Moreover, there is a relationship between eigenvalues of $a$ and $b$, which we state in the following elementary lemma.
\begin{lem}\label{simetriaespectros}
If $\lambda \ne 0,1$ is an eigenvalue of $b$, then $1-\lambda$ is an eigenvalue of $a$, and the operator 
$x|_{N(b-\lambda 1_{N(P)})}$ maps $N(b-\lambda 1_{N(P)})$ isomorphically onto $N(a-(1-\lambda) 1_{R(P)})$. Thus, in particular, these eigenvalues have the same multiplicity. Moreover,
$$
xP_{N\left(b-\lambda 1_{N(P)}\right)}=P_{N\left(a-(1-\lambda) 1_{R(P)}\right)}x.
$$
\end{lem}
\begin{proof} 
Let $\xi\in\h$, $\xi\ne 0$, such that $b\xi=\lambda \xi$ (with $\lambda \ne 0,1$). Then,  by the third relation in (\ref{relaciones proyeccion}),  one has
$$
x\xi=ax\xi+xb\xi=ax\xi+\lambda x\xi,  \hbox{ i.e. },  ax\xi=(1-\lambda)x\xi.
$$
Also note that
$$
N(x)=N(x^*x)=N(b-b^2)=N(b)\oplus N(b-1_{N(P)}),
$$ 
and thus $x\xi\ne 0$ is an eigenvector for $a$, with eigenvalue $1-\lambda$, and the map $x|_{N(b-\lambda 1_{N(P)})}$ is injective from $N(b-\lambda 1_{N(P)})$ to $N\left(ax-(1-\lambda) 1_{R(P)}\right)$. Therefore 
$$
\dim\left(N(b-\lambda 1_{N(P)})\right)\le \dim\left(N(a-(1-\lambda) 1_{R(P)}\right).
$$
By a symmetric argument, using $x^*$ (and the relation $bx^*+x^*a=x$), one obtains equality of these dimensions.

Pick now an arbitrary $\xi\in N(P)$ and write $\xi=\xi_1+\xi_2$, with $\xi_1\in\ N(b-\lambda 1_{N(P)})$ and $\xi_2\perp N(b-\lambda 1_{N(P)})$.
Then 
$$
xP_{N(b-\lambda 1_{N(P)})}\xi=x\xi_1.
$$
On the other hand
$$
P_{N\left(a-(1-\lambda) 1_{R(P)}\right)}x\xi_1=x\xi_1,
$$
by the fact proven above. Let us see that $P_{N\left(a-(1-\lambda)1_{R(P)}\right)}x\xi_2=0$, which will prove our claim. Since $\xi_2\perp N(b-\lambda 1_{N(P)})$,  it follows that $\xi_2=\sum_{l\ge 2}\eta_l+\eta_0+\eta_1$, where $\eta_l$, $l\ge 2$, are eigenvectors of $b$ corresponding to eigenvalues $\lambda_l$ different from $0$, $1$ and $\lambda$, $\eta_0\in N(b)$, $\eta_1\in N(b-1_{N(P)})$ (where these two latter may be trivial).
Note  then that  $\eta_0,\eta_1\in N(x)$, and thus
$$
x\xi_2=\sum_{l\ge 2}x\eta_l,
$$
where the (non nil) vectors $x\eta_l$ are eigenvectors of $a$ corresponding to eigenvalues $1-\lambda_l$, different from $0,1$ and $1-\lambda$.   Thus $P_{N\left(a-(1-\lambda) 1_{R(P)}\right)}x\xi_2=0$.
\end{proof}
For an operator $T\in\b(\h)$, let $r(T)=\dim\left(R(T)\right)$ be the rank of $T$ and $n(T)=\dim\left(N(T)\right)$ the nullity of $T$.
\begin{rem}
This result implies that we may write $a$ and $b$ as
\begin{equation}\label{a,b}
\begin{array}{l}
a=\sum_{n\ge 1} \lambda_n P_n + E_1 \\
b=\sum_{n\ge 1} (1-\lambda_n)P'_n  + E_1' \\ 
\end{array}
\end{equation}
where $1>\lambda_n>0$ is a decreasing set, which may be finite  or a sequence converging to $0$, 
$$
r(P_n)=r(P'_n)<\infty \ , E_1=P_{N(a-1_{R(P)})}\  \hbox{ with } \  r(E_1)<\infty,  \ \hbox{ and } E_1'=P_{N(b-1_{N(P)})}. 
$$
Accordingly, the decomposition of the (non selfadjoint) operator $x$ in singular values is
$$
x=\sum_{n\ge 1} \alpha_n (\sum_{j=1}^{k_n}\xi_{n,j}\otimes \xi'_{n,j}),
$$
where $\alpha_n=\sqrt{\lambda_n-\lambda_n^2}$, and   $\{\xi_{n,j}: 1\le j\le k_n\}$ and $\{\xi'_{n,j}: 1\le j\le k_n\}$ are orthonormal systems which span $R(P_n)$ and $R(P'_n)$, respectively.
\end{rem}

Let us first sort out the pairs where either $P$ or $Q$ have finite rank, and which clearly belong to $\c$. 

\begin{rem}\label{caso finito}

\begin{enumerate}
\item
If $r(P)=k<\infty$, then $(P,Q)\in\c$ for any $Q\in\p$. Moreover, if $\left(P(t),Q(t)\right)$ is a continuous path in $\c$, $P(t)$ and $Q(t)$ are continuous paths in $\p$, and thus $r\left(P(t)\right)=k$ and $r\left(Q(t)\right)=l\le \infty$ along the path. If   $l=\infty$, then $n(Q(t))=n(Q)=m\le\infty$ for all $t$. Conversely, suppose that $(P',Q')$ is another pair in $\c$ with $r(P)=r(P')$, $r(Q)=r(Q')$ and $n(Q)=n(Q')$. Then there exists a unitary operator $U$ such that $UPU^*=P'$.  Let $U=e^{iX}$ for some $X^*=X$. Then 
$$
\alpha(t)=(e^{itX}Pe^{-itX},e^{itX}Qe^{-itX})
$$
is a continuous path inside $\c$ with $\alpha(0)=(P,Q)$ and $\alpha(1)=(P',UQU^*)$.  It is clear that $UQU^*$ and $Q'$ have also the same rank and nullity. Thus,  there exists a unitary operator $W=e^{itY}$ (with $Y^*=Y$) such that $WUQ(WU)^*=Q$. Then
$$
\beta(t)=(P',e^{itY}UQU^*e^{-itY})
$$
is another continuous path inside $\c$ joining $(P',UQU^*)$ and $(P',Q')$. It follows that $(P,Q)$ and $(P',Q')$ lie in the same connected component of $\c$. 
\item
If $n(P)=m<\infty$ and $(P,Q)\in\c$, then it must be $r(Q)=l<\infty$. Indeed, in this case $PQ$ is compact and $(1-P)Q$ has finite rank, then $Q=PQ+(1-P)Q$ is compact, i.e. of finite rank. Thus,  a similar analysis as above can be done, for pairs in $\c$ with the second coordinate of rank $l<\infty$.
\item
 We may summarize this information  as follows. Let
$$
\c_0=\{(P,Q)\in\p\times\p: r(P)<\infty \hbox{ or } r(Q)<\infty\}.
$$
Then the connected components of $\c_0$ are the sets
$$
\c_{k, l}^{m,n}=\{(P,Q)\in\p\times \p: r(P)=k, n(P)=l, r(Q)=m, n(Q)=n\},
$$
with $\min\{k,l,m,n\}$ finite.
\end{enumerate}
\end{rem}
In what follows, unless otherwise stated, we shall suppose that both projections in the pair $(P,Q)\in\c$ have infinite rank and nullity.
To describe the pairs $(P,Q)\in\c$, it will be useful to consider the homomorphism onto the Calkin algebra
$$
\pi :\b(\h)\to \b(\h) / \k(\h).
$$
Put $p=\pi(P)$, $q=\pi(Q)$, which are non zero projections in $\b(\h) / \k(\h)$. Write the matrix of $q$ in terms of $p$. Then
$$
q=\left( \begin{array}{cc} 0 & 0 \\  0 & q' \end{array} \right),
$$
where $q'$ is a projection (i.e. a selfadjoint idempotent) in $\b\left(N(P)\right) / \k\left(N(P)\right)$, the Calkin algebra of $N(P)$. In the Calkin algebra there are three (unitary) equivalence classes of projections:  $0$, $1$ and $e\ne 0,1$ ($e=\pi(E)$ for any $E$ with $R(E)$ and $N(E)$ infinite dimensional). 
\begin{defi}\label{clasesP}
Fix $P\in\p(\h)$. Denote 
$$
\c(P)=\{Q\in\p(\h): PQ \hbox{ is compact}\}.
$$
According to the above classification, relative to $P$ there are two classes of projections $Q$ such that $PQ$ is compact.
\begin{enumerate}
\item If $q'=1$:
$$
\c_1(P)=\{Q\in\p(\h): \pi\left((1-P)(1-Q)(1-P)\right)=\pi(1-P)\}.
$$
This means that $\dim\left(N(b)\right)<\infty$. We shall describe this class below. It is the restricted Grassmannian induced by the decomposition $\h=N(P)\oplus R(P)$ (in the usual description of the restricted Grassmannian: $N(P)$ plays the main role).
\item If $q'$ is a proper projection in $\b(N(P))/\k(N(P))$:
$$
\c_\infty(P)=\{Q\in\p(\h): \pi((1-P)(1-Q)(1-P))\ne \pi(1-P), 0\}.
$$ 
We  shall call this  the  class of {\it essential} projections  relative to $P$. We shall see that  the pairs in Example \ref{el ejemplo} belong to this class. 
\end{enumerate}
\end{defi}
\section{The Halmos decomposition}

Given   orthogonal projections $P$ and $Q$ , we shall call the {\it Halmos decomposition } \cite{halmos}  (though it was certainly used before) the following orthogonal decomposition of $\h$,
$$ 
\h_{11}=R(P)\cap R(Q) \ , \ \ \h_{00}=N(P)\cap N(Q) \ , \ \h_{10}=R(P)\cap N(Q) \ , \ \ \h_{01}=N(P)\cap R(Q)
$$
and $\h_0$ the orthogonal complement of the sum of the above. This last subspace is usually called the {\it generic part} of the pair $P, Q$. Note also that
$$
N(P-Q)=\h_{11}\oplus\h_{00}\ , \ \ N(P-Q-1)=\h_{10} \ \hbox{ and } N(P-Q+1)=\h_{01},
$$
so that the generic part depends in fact of the difference $P-Q$. 

Halmos proved that there is an isometric isomorphism between $\h_0$ and a product Hilbert space $\l\times\l$ such that,  in the above decomposition (putting $\l\times\l$ in place of $\h_0$), the projections are 
$$
P=1\oplus 0 \oplus 1 \oplus 0 \oplus \left(\begin{array}{cc} 1 & 0 \\ 0 & 0 \end{array} \right)
$$
and 
$$
Q=1\oplus 0 \oplus 0 \oplus 1 \oplus \left(\begin{array}{cc} C^2 & CS \\ CS & S^2 \end{array} \right),
$$
where $C=\cos(X)$ and $S=\sin(X)$ for some operator $0\le X\le \pi/2$ in $\l$ with trivial nullspace.

Let us describe the pairs in $\c$ in terms of this decomposition. It should be noted that  the operator $X$ and the space $\l$ are uniquely determined up to unitary equivalence.
\begin{prop}
The pair $(P,Q)$ belongs to $\c$ if and only if  $\h_{11}$ is finite dimensional and $C=\cos(X)$ is compact.
\end{prop}
\begin{proof}
By direct computation,
$$
PQ=1\oplus 0 \oplus 0 \oplus 0 \oplus \left(\begin{array}{cc} C^2 & CS \\ 0 & 0 \end{array} \right).
$$
If $C$ is compact, then $C^2$ and $CS$ are compact. If,  additionally,  $\dim\h_{11}<\infty$, then it is clear that $PQ$ is compact. 

Conversely, if $PQ$ is compact, then clearly  $\dim\h_{11}<\infty$. If the matrix operator
$$ 
\left(\begin{array}{cc} C^2 & CS \\ 0 & 0 \end{array} \right)
$$
is compact, then its $1,1$ entry is also compact. The square root of a positive compact operator is compact (recall that $C\ge 0$), thus $C$ is compact.
\end{proof}
\begin{rem}
If $(P,Q)\in\c$, then the spectral resolution of $X$ can be easily described. Since $0<\cos(X)$ is compact, it follows that
$$
X=\sum_n \gamma_n P_n+\frac{\pi}{2}E,
$$
where $0<\gamma_n< \pi/2$ is an increasing (finite or infinite) sequence, and  $P_n$, $E$ are the projections onto the eigenspaces associated to $\gamma_n$ and and $\pi/2$, respectively . For all $n$, $\dim R(P_n)<\infty$, and
$$
R(E)\oplus\left(\oplus_{n\ge 1} R(P_n)\right)=\l.
$$

\end{rem}
From the spectral picture of $X$ above, one obtains the following result, which states that in their generic part, all pairs in $\c$ are obtained as in Example \ref{ejemploClaseE}
\begin{prop}
The pair $(P,Q)$ belongs to $\c$ if and only if  the following conditions are satisfied:
\begin{enumerate}
\item
$\h_{11}$ is finite dimensional.
\item
The subspaces  $M=P(\h_0)$ and $N=Q(\h_0)$  of the generic part $\h_0$, satisfy
$$
M\oplus N=\h_0.
$$
\item
The idempotent $E=P_{M\parallel N}$ corresponding with this decomposition has matrix form, in terms of its range $M$,
$$
E=\left(\begin{array}{cc} 1 & K \\ 0 & 0 \end{array} \right)
$$ 
for $K:N\to M$ a compact operator.
\end{enumerate}
\end{prop}
\begin{proof}
It is clear that these conditions imply that  $(P,Q)\in\c$. It is also clear  that  condition $\dim(\h_{11})<\infty$ is necessary. Denote by $P_0$ and $Q_0$ the reductions of $P$ and $Q$ to their generic part. In Halmos' model $\h_0=\l\times\l$,  $\sin(X)$ is invertible in $\l$. Then (reasoning with matrices in terms of $\h_0=\l\times \l$)
$$
P_0-Q_0=\left(\begin{array}{cc} S^2 & CS \\ CS & -S^2 \end{array} \right) \ \hbox{ and thus } \  (P_0-Q_0)^2=\left(\begin{array}{cc} S^2 & 0 \\ 0 & S^2 \end{array} \right)
$$ is invertible. Then $P_0-Q_0$ is invertible. This means that $M\oplus N=\h_0$ (see \cite{buckholtz}). Moreover, by a formula in \cite{buckholtz}, and after straightforward computations, it holds that
$$
E=P_{M\parallel N}=P_0(P_0-Q_0)^{-1}=\left(\begin{array}{cc} 1 & -CS^{-1} \\ 0 & 0 \end{array} \right).
$$
Note that $-CS^{-1}$ is compact in $\l$. 
\end{proof}
We shall describe the different subclasses of $\c$ in terms of the Halmos decomposition and the spectral resolution of $X$. The  class $\c_0$ is easiest to describe. Recall that $(P,Q)\in\c_0$ if $\dim R(Q)$ or $\dim R(P)<\infty$. 
\begin{prop}
Let $(P,Q)\in\c$. Then $(P,Q)\in\c_0$ if  and only if the sequence $\{\gamma_n\}$ is finite, $\dim \h_{01}<\infty$ and $\dim R(E_1)<\infty$. 
\end{prop}

\section{A spatial characterization}

In this section we briefly address the following question: let $\s$ and $\t$ be closed subspaces of $\h$, when is $P_\s P_\t$ compact?

\begin{teo}
$P_\s P_\t$ is compact if and only if there exist orthonormal bases $\{\xi_n: n\ge 1\}$ and  $\{\psi_n: n\ge 1\}$ of $\s$ and $\t$ respectively, such that $\langle\xi_n,\psi_k\rangle=0$ if $n\ne k$ and $\langle\xi_n, \psi_n\rangle\to 0$ {\rm(}$n\to \infty${\rm)}.
\end{teo}
\begin{proof}
The sufficiency of this condition is clear. If $\{\xi_n: n\ge 1\}$ and $\{\psi_n: n\ge 1\}$ are bi-orthogonal and $\langle\xi_n,\psi_n\rangle\to 0$, then
$$
P_\s P_\t=(\sum_{n\ge 1} \langle\ , \xi_n\rangle\xi_n)(\sum_{k\ge 1} \langle\ , \psi_k\rangle\psi_k)=\sum_{n\ge 1}\langle\xi_n,\psi_n\rangle\xi_n\otimes\psi_n.
$$
This is essentially the  singular value decomposition for $P_\s P_\t$. Indeed, put 
$$
\langle\xi_n,\psi_n\rangle=e^{i\theta_n}|\langle\xi_n,\psi_n\rangle|
$$ 
and replace $\xi'_n=e^{-i\theta_n}\xi_n$. Then
$$
P_\s P_\t=\sum_{n\ge 1}|\langle\xi'_n,\psi_n\rangle|\xi'_n\otimes\psi_n
$$
with singular values $|\langle\xi'_n,\psi_n\rangle|\to 0$, and thus $P_\s P_\t$ is compact.

Conversely, suppose that $T=P_\s P_\t$ is compact. 
Then clearly $\overline{R(T)}\subset\s$ and $N(P_\t)=\t^\perp\subset N(T)$, i.e. $N(T)^\perp\subset\t$. Thus $P_\s P_{\overline{R(T)}}=P_{\overline{R(T)}}$ and $P_\t P_{N(T)^\perp}=P_{N(T)^\perp}$.
Then
$$
T=P_{\overline{R(T)}}TP_{N(T)^\perp}=P_{\overline{R(T)}}P_\s P_\t P_\s P_{N(T)^\perp}=P_{\overline{R(T)}} P_{N(T)^\perp}.
$$

Consider the singular value decomposition of $T$:
$$
T=\sum_{n\ge 1} s_n \xi_n\otimes \psi_n,
$$
where $\{\xi_n: n\ge 1\}$ and $\{\psi_n: n\ge 1\}$ are orthonormal bases of $\overline{R(T)}$ and $N(T)^\perp$, respectively.
Note that $T$ satifies the equation 
\begin{equation}\label{crimmings}
T^2=P_\s P_\t P_\s P_\t=TT^*T.
\end{equation}
 Straightforward computations show that
$$
T^2=\sum_{n\ge 1} s_nT\xi_n \otimes\psi_n=\sum_{n,k\ge 1} s_ns_k\langle\xi_n,\psi_k\rangle\xi_k\otimes \psi_n
=\sum_{k\ge 1}\langle\ \ \  , \sum_{n\ge 1} s_ns_k\langle\psi_k,\xi_n\rangle\psi_n\rangle\xi_k.
$$
On the other hand,
$$
T^*T=\sum_{k\ge 1} s_k^2 \psi_k\otimes\psi_k
$$
and 
$$
TT^*T=\sum_{k\ge 1} s_k^2 T\psi_k\otimes \psi_k=\sum_{k\ge 1}s_k^3\xi_k\otimes \psi_k=\sum_{k\ge1} s_k^3\langle\ \ \ , \psi_k\rangle\xi_k.
$$
Therefore using (\ref{crimmings})
$$
\sum_{n\ge 1} s_ns_k\langle\psi_k,\xi_n\rangle\psi_n=s_k^3\psi_k.
$$
Then
$\langle\xi_n,\psi_k\rangle=0$ if $n\ne k$ and $\langle\xi_n,\psi_n\rangle=s_n$.

Let us extend  the orthonormal bases $\{\xi_n\}$ and $\{\psi_n\}$ of $\overline{R(T)}$ and $N(T)^\perp$ to orthonormal bases of $\s$ and $\t$. Note that if $\xi\in\s\ominus\overline{R(T)}$ and $\psi\in\t\ominus N(T)^\perp$, then
$$
\langle\xi,\psi\rangle=\langle P_\s\xi,P_\t\rangle=\langle\xi, P_\s P_\t\psi\rangle=\langle\xi,T\psi\rangle=0,
$$
because $\xi\perp R(T)$. 
Therefore we can extend the bases $\{\xi_n\}$ and $\{\psi_n\}$ to bases $\{\xi'_n\}$ and $\{\psi'_n\}$ arbitrarily, and the properties that $\langle\xi'_n,\psi'_k\rangle=0$ if $n\ne k$,  and  that $\langle\xi'_n,\psi'_n\rangle$ ($=s_n$ or $0$) converges to $0$ remain valid for the extended bases. 
\end{proof}
\begin{rem}
Equation (\ref{crimmings}) above in fact characterizes operators which are the product of two projections, compact or not. This was shown by Crimmins, and published  in \cite{crimm}.
\end{rem}

\begin{rem}
In the special case $\s=R(P_I)$ and $\t=R(Q_J)$ for $I,J\subset\mathbb{R}^3$ of finite Lebesgue measure, the bi-orthogonal system is given by the so-called {\it prolate spherical functions}, and their (normalized) images under $Q_J$ (see, for instance, \cite{bell labs}). That these functions are bi-orthogonal (or double orthogonal, as stated in \cite{bell labs}) is well known. As seen above, this is not a special feature of this example but a general property when $PQ$ is compact.
\end{rem}
\section{Unitary actions}

We shall use two unitary actions to describe the structure of $\c$. The full unitary group $\u(\h)$ acts on pairs in $\c$ by joint inner conjugation:
$$
U\cdot (P,Q)=(UPU^*,UQU^*), \ \ U\in\u(\h), \ \ (P,Q)\in\c.
$$ 
We shall make use of another local unitary action, on pairs in $\c$ with the first coordinate $P_0$ fixed.  Recall the definition  of the restricted unitary group, where the restriction is given by the decomposition $\h=R(P_0)\oplus N(P_0)$ (see \cite{pressleysegal}),
$$
\u_{res}(P_0)=\{U\in\u(\h): [U,P_0]\in\k(\h)\}.
$$
In matrix form, in terms of the given decomposition, these unitaries are of the form
$$
U=\left( \begin{array}{cc} u_{11} & u_{12} \\ u_{21} & u_{22} \end{array} \right)
$$
where $u_{12}$ and $u_{21}$ are compact operators. Elementary matrix computations, involving the  fact that $U$ is unitary, imply that $u_{11}$ and $u_{22}$ are Fredholm operators in $R(P_0)$ and $N(P_0)$, respectively, and that 
$$
ind (u_{22})=-ind (u_{11}).
$$
The integer $ind (u_{11})$ is usually called the {\it index} of $U$. It is known that this index parametrizes the connected components of $\u_{res}(P_0)$: two unitaries $U,W\in\u_{res}(P_0)$ belong to the same connected component if and only if $ind (U)= ind (W)$ (see for instance \cite{pressleysegal} or \cite{carey}).

Let us prove that $\u_{res}(P_0)$ acts by inner conjugation of the classes $\c_x(P_0)$ ($x=0,1,\infty$).
\begin{prop}
Let $Q\in\c_x(P_0)$ {\rm(}$x=0,1,\infty${\rm)} and $U\in \u_{res}(P_0)$. Then $UQU^*\in\c_x(P_0)$.
\end{prop}
\begin{proof}
Straightforward matrix computation:
$$
UQU^*=\left( \begin{array}{cc} u_{11} & u_{12} \\ u_{21} & u_{22} \end{array} \right)\left( \begin{array}{cc} a & x \\ x^* & y \end{array} \right)\left( \begin{array}{cc} u_{11}^* & u_{21}^* \\ u_{12}^* & u_{22}^* \end{array} \right).
$$
The $1,1$ entry of this product is
$$
u_{11}au_{11}^*+u_{12}x^*u_{11}^*+u_{11}xu_{12}^*+u_{12}bu_{12}^*,
$$
where $a, x$ and $u_{12}$ are compact, therefore the $1,1$ entry is compact.
The $1,2$ entry is
$$
u_{11}au_{21}^*+u_{12}x^*u_{21}^*+u_{11}xu_{22}^*+u_{12}bu_{22}^*,
$$ 
which is compact by a similar argument. Then $P_0UQU^*\in\k(\h)$.
\end{proof}
We shall mainly use  $\u_{res}^0(P_0)$, the  connected component of the identity (or zero index component). This component is an exponential group, namely
$$
\u_{res}^0(P_0)=exp \{iX\in\b(\h): X^*=X, [X,P_0]\in\k(\h)\},
$$
(see \cite{carey}, \cite{pressleysegal}).
Note that $\u_{res}(P_0)$ is the unitary group of the C$^*$-algebra 
$\a_{P_0}(\h)$ of operators $T$ in $\h$ such that $[T,P_0]\in\k(\h)$.

\section{The restricted Grassmannian}

Let us recall some elementary facts concerning the restricted Grassmannian of a decomposition of $\h=\n_0\oplus\n_0^\perp$. Denote by $E_0$ the orthogonal projection onto $\n_0$.

\begin{defi} {\rm \cite{segalwilson}}

A projection $Q$ belongs to the restricted Grassmannian $\p_{res}(\n_0)$ with respect to the decomposition $\h=\n_0\oplus\n_0^\perp$, or more precisely, with respect to subspace $\n_0$, if and only if
\begin{enumerate}
\item
$$
E_0Q|_{R(Q)}:R(Q)\to \n_0\in\b\left(R(Q),\n_0\right)
$$
is a Fredholm operator in $\b(R(Q),\n_0)$, and
\item
$$
(1-E_0)Q|_{R(Q)}:R(Q)\to \n_0^\perp\in\b\left(R(Q),\n_0^\perp\right)
$$
is compact.
\end{enumerate}
\end{defi}
The index of the first operator characterizes the connected components of $\p_{res}(\n_0)$.
The following result is elementary.
\begin{lem}\label{lema sato}
Let $Q\in\p(\h)$ with matrix {\rm(}in terms of $\h=\n_0\oplus\n_0^\perp${\rm)}
$$
Q=\left( \begin{array}{cc} a & x\\ x^* & b \end{array} \right).
$$
Then $Q\in \p_{res}(\n_0)$ if and only if $a$ is Fredholm in $\b(\n_0)$, and $b$ and $x$ are compact.
\end{lem}
\begin{proof}
The proof is based on the following elementary facts:
\begin{itemize}
\item
$A\in\b(\h_1,\h_2)$ is a Fredholm operator if and only if   $AA^*$ is a Fredholm operator in $\h_1$ and $N(A)$ is finite dimensional.
\item
$A\in\b(\h_1,\h_2)$ is compact if and only if $A^*A\in\b(\h_1)$ is compact.
\end{itemize}
Suppose first that $Q\in\p_{res}(\n_0)$. Then $E_0Q\in\b\left(R(Q),\n_0\right)$ is Fredholm,  and thus
$$
E_0Q(E_0Q)^*|_{\n_0}=E_0QE_0|_{\n_0}=a
$$
is Fredholm in $\n_0$. Also $(1-E_0)Q\in\b(R(Q),\n_0^\perp)$ is compact, and thus
$$
(1-E_0)Q(1-E_0-Q)^*|_{\n_0^\perp}=(1-E_0)Q(1-E_0)|_{\n_0^\perp}=b
$$
is compact in $\n_0^\perp$. The fact that $Q$ is a projection implies the relation
$b-b^2=x^*x$, and thus $x$ is compact.

Conversely, by the last computations, if $x$ and $b$ are compact,
then $(1-E_0)Q\in\b\left(R(Q),\n_0^\perp\right)$ is compact.
Similarly, $E_0Q(E_0Q)^*|_{\n_0}=a$ is Fredholm,
thus $E_0Q$, as an operator in $\b(R(Q),\n_0)$,
has closed range (equal to the range of $a$) with finite codimension.
Let us prove that its nullspace is finite dimensional.
Let $\xi=\xi_++\xi_-=Q\xi$ such that $E_0\xi=0$, ($\xi_+\in\n_0$, $\xi_-\in\n_0^\perp$).
This implies that
$$
\left\{ \begin{array}{l} \xi_+ =a \xi_++x \xi_- \\  \xi_-=x^* \xi_++b \xi_- \end{array} \right.
$$
and $\xi_+=0$.
The second equation then reduces to $\xi_-=b\xi_-$, i.e., $\xi_-$ lies in the $1$-eigenspace of the compact operator $b$. Thus $\xi_-$ lies in a finite dimensional space. It follows that $N(E_0Q|_{R(Q)})$ is finite dimensional.
\end{proof}
\begin{coro}
Let $P\in\p(\h)$ be such that  $N(P),R(P)$ are  infinite dimensional. Then $\c_1(P)$ coincides with the restricted Grassmannian of $\h$ induced by the decomposition $\h=N(P)\oplus R(P)$.
\end{coro}
\begin{proof}
In the description of the classes $\c_i(P)$ at Definition \ref{clasesP} (given in matrix form in terms of the decomposition $\h=R(P)\oplus N(P)$, note the reversed order),  a projection $Q$ belongs to $\c_1(P)$ if and only if, in the Calkin algebra, its $2,2$ entry is the identity  and all other entries are nil. By the above Lemma, this means that $Q$ belongs to the restricted Grassmannian of the decomposition $\h=N(P)\oplus R(P)$.
\end{proof}
From now on we shall refer  this set of projections  as the restricted Grassmannian of $N(P_0)$. 
\begin{rem} 
The group $\u_{res}^0(P_0)$ acts transitively on the connected components of $\c_1(P_0)$, which are parametrized by the Fredholm index defined in the restricted Grassmannian of $N(P_0)$.
\end{rem}
Let us denote by
$$
\c_1=\{(P,Q)\in\c: Q\in\c_1(P)\},
$$
the union of $\c_1(P)$ for all $P\in\p_{\infty}(\h)$,
where $\p_{\infty}(\h)$ denotes the (connected) space of projections in $\h$ with infinite dimensional range and nullspace.
\begin{teo} 
The connected components of $\c_1$ are parametrized by the Fredholm index. Namely, $(P,Q), (P',Q')\in\c_1$ lie in the same connected component if  and only if 
the index of $Q$ in the restricted Grassmannian of $N(P)$ coincides with the index of $Q'$ in the restricted Grassmannian of $N(P')$.
\end{teo}
\begin{proof}
There exists a unitary operator $U\in\u(\h)$ such that $U^*P'U=P$. Consider the pair $U^*\cdot(P',Q')=(P,U^*Q'U)$. Note that  $(P,U^*Q'U)$ belongs to the restricted Grassmannian of $N(P)$, and it has the same index as $(P',Q')$. Since $\u(\h)$ is connected, this means that one is reduced to the case $P=P'$, where the result is valid due to the above Corollary.
\end{proof}

Note that the class $\c_1$ can be described in terms of the Halmos decomposition.
\begin{prop}
Let $(P,Q)\in\c$. Then
the following are equivalent:
\begin{enumerate}
\item
 $(P,Q)\in\c_1$.
 \item
 $\dim \h_{00}<\infty$. 
 \item
 $\dim N(b)<\infty$.
 \end{enumerate}
 
In this case, the index of $Q$ in the restricted Grassmannian of $N(P)$ is given by
$$
\dim \h_{01}-\dim \h_{10}.
$$
\end{prop}
\begin{proof}
The (five space) Halmos decomposition induces a (four space) decomposition of $\h$ which reduces both $P$ and $Q$. Namely,
$$
\h=\h_{00} \oplus \h_{11} \oplus \h' \oplus \h_0,
$$
where $\h'=\h_{10}\oplus\h_{01}$.
By Lemma (\ref{lema sato}), the part of $Q$ which acts on $N(P)$ must be a Fredholm operator. By the above reduction, this amounts to show that both
$0$ acting in $\h_{00}$  and $S^2$ acting in the space $\l$,  are Fredholm operators (recall notations from Section 3). The first assertion means that $\dim \h_{00}<\infty$. With respect to the second, 
$$
S^2=\sin^2(X)=\sum_{n} \sin^2(\gamma_n)P_n +E 
$$
is always a Fredholm operator (recall that $N(S)=N(X)=\{0\}$), since $0<\sin^2(\gamma_n)$ is a finite set or a sequence increasing to $1$.
If $Q\in\c_1(P)$, then $b$ is Fredholm in $N(P)$, and thus $N(b)$ is finite dimensional. Conversely, the fact that $Q\in\c(P)$ implies that the spectral decomposition of $b$ is of the form
$$
b=\sum_{n\ge 1} (1-\lambda_n)P_n' + E_1',
$$
with $1>\lambda_n>0$ a finite set or a strictly decreasing sequence converging to $0$. If $N(b)$ is finite dimensional, then  $b$ is a Fredholm operator, and thus $Q$ belongs to the restricted Grassmannian of $N(P)$, i.e., $Q\in\c_1(P)$.

If $Q$ lies in the restricted Grassmannian, it is well known that the index of $Q$ with respect to $N(P)$
is 
$$
\dim \left(R(Q)\cap N(P)\right)-\dim\left(N(Q)\cap R(P)\right)=\dim \h_{01}-\dim \h_{10}.
$$ 
\end{proof}  

\begin{ejem}
Let us return to Example 1.1.2: 
$$
\h=L^2(\mathbb{T},dt)=\h_-\oplus\h_+, \  P=P_{\varphi\h_+}^\perp , \ Q=P_{\psi\h_+}, 
$$
where $\h_+$ is the Hardy space and $\varphi,\psi:\mathbb{T}\to \mathbb{T}$  are continuous.
The $2,2$ entry $b$ of $Q$ in terms of $P$ is unitarily equivalent to
$$
P_+M_{\bar{\varphi}\psi}P_+M_{\bar{\psi}\varphi}P_+=(P_+M_{\bar{\psi}\varphi}P_+)^*P_+M_{\bar{\psi}\varphi}P_+.
$$
Note that $P_+M_{\bar{\psi}\varphi}|_{\h_+}$ is a Toeplitz operator with non vanishing continuous symbol, therefore $b$ is a Fredholm operator \cite{douglas}, and $(P,Q)\in\c_1$. The index of the pair is (minus) the winding number of the symbol $\bar{\psi}\varphi$ \cite{douglas}. Subspaces $\varphi\h_+$ with $\varphi$ continuous and non vanishing were studied in \cite{pressleysegal} and \cite{segalwilson} in connection with parametrizations of solutions of the KdV equation.
\end{ejem} 

\section{Essential projections}
Following the notation of the previous section, denote
$$
\c_\infty=\{(P,Q)\in\c: Q\in\c_\infty(P)\},
$$
the union of $\c_\infty(P)$ for all $P\in\p_\infty(\h)$.
Let $(P,Q)\in\c_\infty$. Write $Q$ as a matrix in terms of $P$ as before, 
$$
Q=\left( \begin{array}{cc} a & x \\ x^* & b \end{array} \right)
$$
with $a=\sum_{n\ge 1} \lambda_n P_n + E_1$ and $b=\sum_{n\ge 1} (1-\lambda_n)P'_n + E_1'$.
Define
$$
Q_d=\left( \begin{array}{cc} E_1 & 0 \\ 0 & \sum_n P'_n+E'_1 \end{array} \right).
$$
Note that $Q_d$ is a projection;  it is also clear that $Q_d\in\p_\infty(\h)$. Indeed, since $r(E_1)<\infty$, it follows that $\dim N(Q_d)=\infty$. If the sequence $\{\lambda_n\}$ is finite, the facts that they have finite multiplicities and that $b$ is a Fredholm operator, imply that $r(E'_1)=\infty$. If the sequence is infinite, then $r(\sum_n P'_n)=\infty$. In any case, $r(Q_d)=\infty$.  
\begin{lem}
$B=Q+Q_d-1$ is invertible in $\h$.
\end{lem}
\begin{proof}
Let $N$ (resp. $N'$) denote the orthogonal projection onto $N(a)$ (resp. $N(b)$) in $R(P)$ (resp. $N(P)$), and write
 $1_{R(P)}=E_1+N+\sum_{n\ge 1}P_n$,  and $1_{N(P)}=E'_1+N'+\sum_{n\ge 1}P'_n$. One has
$$
B=\left( \begin{array}{cc} \sum_{n\ge 1} (\lambda_n-1)P_n +E_1-N & x \\ x^* & \sum(1-\lambda_n)P'_n+E'_1-N' \end{array} \right).
$$
The diagonal entries of $B$ are invertible in $R(P)$ and $N(P)$. Indeed, they are diagonal operators with  non nil eigenvalues that  accumulate (eventually) at $-1$ and $1$, respectively. The codiagonal entries of $B$ are compact. It follows that $B$ is of the form invertible plus compact. Thus it is a Fredholm operator, and in particular it has closed range. Therefore, since $B$ is selfadjoint,  it suffices to show that it has trivial nullspace. Note that $B$ is a difference of 
projections, namely
$$
B=Q-(1-Q_d).
$$
It is an elementary fact that the nullspace of a difference of projections is
$$
N(B)=\left(N(Q)\cap N(1-Q_d)\right)\oplus \left(R(Q)\cap R(1-Q_d)\right)=\left(N(Q)\cap R(Q_d)\right)\oplus \left(R(Q)\cap N(Q_d)\right).
$$
Let us  see that $N(Q)\cap R(Q_d)=\{0\}$. Let $\xi+\eta\in R(P)\oplus N(P)=\h$ in $N(Q)\cap R(Q_d)$. Then
\begin{equation}\label{nucleo de B 1}
E_1\xi=\xi  \ \  \hbox{ and }\ \ \sum_{n\ge 1}P'_n\eta+E'_1\eta=\eta.
\end{equation}
This implies that $P_n\xi=0$ for all $n$, $N\xi=0$ and $N'\eta=0$. Also one has
\begin{equation}\label{nucleo de B 2}
\left\{\begin{array}{l} 0=\sum_{n\ge 1}\lambda_nP_n\xi+E_1\xi+x\eta=\xi+x\eta \\ 0=x^*\xi+\sum_{n\ge 1}(1-\lambda_n)P'_n\eta+E'\eta. \end{array} \right.
\end{equation}
Recall that $R(x)=\oplus_{n\ge 1}R(P_n)$ which is orthogonal to $R(E_1)$. Thus $\xi=0$ and $x\eta=0$. Since the nullspace of $x$ is $R(N')\oplus R(E'_1)$, one has that $\eta=N'\eta+E'_1\eta$. Combining this with the second equality in (\ref{nucleo de B 1}), one gets  $N'\eta=0$ and $E'\eta=\eta$ (and $P'_n\eta=0$ for all $n$). Using these facts in the second equation of (\ref{nucleo de B 2}), one obtains $\eta=0$. 

The fact that $R(Q)\cap N(Q_d)=\{0\}$ is proved in a similar fashion.
\end{proof}

\begin{rem}
Buckholtz  \cite{buckholtz} proved that a difference of projections $P_1-P_2$ is invertible if and only if $\|P_1+P_2-1\|<1$. In our case, this implies that 
$$
\|Q-Q_d\|<1.
$$
\end{rem} 

\begin{lem}
The unitary part $U$ of $B$ in the polar decomposition $B=U|B|$ belongs to $\u_{res}^0(P)$.
\end{lem}
\begin{proof}
As remarked in the above proof, the off-diagonal entries of $B$ (in its matrix in terms of $P$) are compact. Therefore $B$ is an invertible element in the C$^*$-algebra $\a_P(\h)$. It follows that its unitary part is a unitary element of this algebra, namely $\u_{res}(P)$. We need to show that it has index zero. The index is in fact defined in the whole invertible group of $\a_P(\h)$, and it coincides with the index of the $1,1$ entry. As it was also pointed out in the proof above, the $1,1$ entry of $B$ is invertible in $R(P)$, and thus it has trivial index.
\end{proof}

\begin{rem}
It is well known (see for instance \cite{cpr}) that if an invertible operator intertwines two selfadjoint projections, then its unitary part in the polar decomposition also does. In our case,
$$
BQ=QQ_d=Q_dB.
$$ 
Therefore $UQ=Q_dU$, or $UQU^*=Q_d$.

Note also the fact that since $B$ is selfadjoint, $U$ is a {\it symmetry} (i.e., a selfadjoint unitary: $S^*=S^{-1}=S$), so that $UQU=Q_d$.
\end{rem}

\begin{lem}
Let $E, F$ be two projections in $\p_\infty(\h)$ which commute  with $P$. Then they are unitarily equivalent with a unitary operator in $\u_{res}^0(P)$.
\end{lem}
\begin{proof}
In terms of $P$, one has
$$
E=\left(\begin{array}{cc} E_1 & 0 \\ 0 & E_2 \end{array} \right)\  \hbox{ and } \  F=\left(\begin{array}{cc} F_1 & 0 \\ 0 & F_2 \end{array} \right),
$$
where $E_1$ and $F_1$ have finite rank in $R(P)$ and $E_2$ and $F_2$ have infinite rank and nullity in $N(P)$. By means of  a unitary operator of the form
$$
\left(\begin{array}{cc} 1_{R(P)} & 0 \\ 0 & W \end{array} \right),
$$
one is reduced to the case $E_2=F_2$. Clearly this unitary operator belongs to $\u_{res}^0(P)$. In order to prove that $E$ and $F$ are conjugate with a unitary in $\u_{res}^0(P)$, it suffices to show that any of these projections, for instance $E$,  can be conjugated with 
$$
E_0=\left(\begin{array}{cc} 0 & 0 \\ 0 & E_2 \end{array} \right).
$$
Consider the following orthonormal bases:
\begin{itemize}
\item
$\{e_n: 1\le n\}$  an orthonormal basis of $R(E_2)$ (in $N(P)$).
\item
$\{e'_l: 1\le l\}$  an orthonormal basis of $N(P)\ominus R(E_2)$.
\item
$\{f_k: 1\le k \}$  an orthonormal basis of $R(P)$,
 with $f_1,\dots , f_N$ spanning $R(E_1)$.
 \end{itemize}
 Consider $U$ defined as follows:
 \begin{itemize}
 \item
 $U(e_n)=f_n$ if $1\le n\le N$, and  $U(e_n)=e_{n-N}$ if $n\ge N+1$.
 \item
 $U(e'_l)=e'_{l+N}$.
\item
$U(f_k)=e'_k$ if $1\le k\le N$, and $U(f_k)=f_k$ if $n\ge N+1$.
\end{itemize}
It is straightforward  that $U$ is a unitary operator. Note also that $U$ is not the identity only on a finite number of $f_k$, and thus 
$UP$ and $PU$ are of the form $P$ plus compact. Therefore $[U,P]$ is compact, i.e. $U\in\u_{res}(P)$. For the same reason, on $R(P)$, $U$ is the identity plus a finite rank operator, and thus $U$ has index zero. Finally, by construction,
$$
U\left(R(E_0)\right)=R(E_1) \hbox{ and } U\left(N(E_0)\right)=N(E_1).
$$
\end{proof}
From these facts, the main result of this section follows.

\begin{teo}
Let $P_0\in\p_\infty(\h)$.

\begin{enumerate}
\item
The action of $\u_{res}^0(P_0)$ is transitive in $\c_\infty(P_0)$. In particular, $\c_\infty(P_0)$ is connected.
\item
$\c_\infty$ is connected,
\end{enumerate}
\end{teo}
\begin{proof}
Let $Q$ and $R$ be elements of $\c_\infty(P_0)$. By the first two lemmas above, $Q$ is $\u_{res}^0(P_0)$-conjugate to $Q_d$  and $R$ is $\u_{res}^0(P_0)$-conjugated to $R_d$.
$R_d$ and $Q_d$ are $\u_{res}^0(P_0)$-conjugate by the third lemma.

To prove the second assertion, suppose that $(P,Q)$ and $(P',Q')$ belong to  $\c_\infty$. Since by hypothesis $P,P'\in\p_\infty(\h)$, there exists a unitary operator $W=e^{iX}$ (with $X^*=X$) such that $WPW^*=P'$. The pairs $(P,Q)$ and $(P',WQW^*)$ are homotopic in $\c_\infty$ (for instance, by means of the curve
$(e^{itX}Pe^{-itX},e^{itX}Qe^{-itX})$). Thus,  it suffices to show that $(P',WQW^*)$ and $(P',Q')$ are homotopic in $\c_\infty$. This is the first assertion.
\end{proof}

Note that in particular, this implies that if $Q\in\c_\infty(P)$, then also $Q_d\in\c_\infty(P)$. This fact could have been obtained directly from the definition of $Q_d$.

\begin{rem}
Consider the example at the beginning of Section 1, namely let $I, J$ be measurable subsets of $\mathbb{R}^n$ of finite measure, and
put $P_I,Q_J\in\p(L^2(\mathbb{R}^n,dx))$ given by
$$
P_If=\chi_If \ \ \hbox{ and } \ \ Q_Jf= \left(\chi_J \hat{f}\right)\check{\ }.
$$ 
Lenard proved \cite{lenard} that $N(P_I)\cap N(Q_J)$ is infinite dimensional. Therefore, the matrix of $Q_J$ in terms of $P_I$ (whose first column and row are compact) has the $2,2$ entry which is not a Fredholm operator. Clearly,  it is not compact (which would mean that $Q_J$ has finite rank). Therefore  $(P_I,Q_J)\in\c_\infty$.  Moreover, given another pair $I'$, $J'$ of finite Lebesgue measure subsets of $\mathbb{R}^n$, the pairs $(P_I, Q_J)$ and $(P_{I'}, Q_{J'})$ are homotopic in $\c_\infty$.
\end{rem}

The above Remark, showing that pairs in the example by Lenard belong to $\c_\infty$, can be generalized. Recall the characterizations of $\c_0$ and $\c_1$ in terms of the Halmos decomposition.
\begin{prop}
Let $(P,Q)\in\c$. Then $(P,Q)\in\c_\infty$ if and only if  $\dim R(Q)=\infty$ and $\dim \h_{00}=\infty$.
\end{prop}
\begin{proof}
Use the same argument as in the above Remark.
\end{proof}
\begin{rem}
Note that $(P,Q)\in\c$ if and only if $(Q,P)\in\c$: $PQ$ is compact if and only if $QP$ is compact. There is, however,  an abuse of notation in this assertion, because we have supposed from the beginning that the first coordinate of the pair must belong to $\p_{\infty}$. Assume thus that also $Q\in\p_{\infty}$. 

Note  that if  $(P,Q)\in\c_1$, then also $(Q,P)\in\c_1$. This follows in a straightforward manner from the definition of the restricted Grassmannian, by taking adjoints. Also it is clear that the index of the reversed pair changes sign.

As a consequence (since the class $\c_0$ is explicitly excluded), it follows that $(P,Q)\in\c_\infty$ implies that $(Q,P)\in\c_\infty$.
\end{rem}

\section{Regular structure}

Let us recall some basic facts on the differential geometry of the set $\p(\h)$ (see for instance \cite{pr}, \cite{cpr}, \cite{pemenoscu}.
\begin{rem}\label{geometria de P}

\begin{enumerate}
\item
The space $\p(\h)$ is a homogeneous space under the action of the unitary group $\u(\h)$ by inner conjugation. The orbits of the action coincide with the connected components of $\p(\h)$, which are: $\p_{n,\infty}(\h)$ (projections of nullity $n$), $\p_{\infty,n}(\h)$ (projections of rank $n$) and $\p_{\infty}(\h)$ (projections of infinite rank and nullity). These components are $C^\infty$-submanifolds of $\b(\h)$.
\item
There is a natural linear connection in $\b(\h)$. If $\dim \h <\infty$, it is the Levi-Civita connection of the Riemannian metric which consists of considering the Frobenius inner product at every tangent space. It is based on the diagonal - codiagonal decompositon of $\b(\h)$. To be more specific, given $P_0\in\p(\h)$, the tangent space of $\p(\h)$ at $P_0$ consists of all selfadjoint codiagonal matrices (in terms of  $P_0$). The linear connection in $\p(\h)$ is induced by a reductive structure, where the horizontal elements at $P_0$ (in the Lie algebra of $\u(\h)$: the space of antihermitian elements of $\b(\h)$) are the codiagonal antihermitian operators. The geodesics of $\p$ which start at $P_0$ are curves of the form
\begin{equation}\label{geodesica}
\delta(t)=e^{itX}P_0e^{-itX},
\end{equation}
with $X^*=X$ codiagonal with respect to $P_0$. 
It was proved in \cite{pr} that if $P_0,P_1\in\p(\h)$ satisfy $\|P_0-P_1\|<1$, then there exists a unique geodesic (up to reparametrization) joining $P_0$ and $P_1$. This condition is not necessary for the existence of a unique geodesic. 
\item In \cite{pemenoscu} a necessary and sufficient condition was found, in order that there exists a unique geodesic joining two projections $P$ and $Q$. This is the case if and only if
$$
R(P)\cap N(Q)=N(P)\cap R(Q)=\{0\}.
$$
\item
It is sometimes useful to parametrize projections using symmetries $S$ ($S^*=S$, $S^2=1$), via the affine map
$$
P \longleftrightarrow S_P=2P-1.
$$
Some algebraic computations are simpler with symmetries. For instance, the condition that the exponent $X$ (of the geodesic) is $P_0$-codiagonal means that $X$ anti-commutes with $S_{P_0}$. Thus the geodesic (\ref{geodesica}), in terms of symmetries, can be expressed
$$
S_{\delta}(t)=e^{itX}S_{P_0}e^{-itX}=e^{2itX}S_{P_0}=S_{P_0}e^{-2itX}.
$$
\end{enumerate}
\end{rem}

Fix $P_0\in\p_\infty(\h)$. We shall see first that $\c(P_0)$ is a differentiable manifold.
If $\a$ is a C$^*$-algebra, denote by $\a_h$ the space of selfadjoint elements of $\a$.

\begin{lem}
If $Q,Q'\in\c(P_0)$ and $\|Q-Q'\|<1$, then there exists $U\in\u_{res}^0(P_0)$ such that $UQU^*=Q'$. This unitary operator $U$ can be chosen as an explicit smooth formula in terms of $Q$ and $Q'$. In particular, $Q$ and $Q'$ lie in the same {\rm(}class and{\rm)} connected component of $\c(P_0)$.
\end{lem}
\begin{proof}
If $\|Q-Q'\|<1$, then there exists a unique geodesic joining $Q$ and $Q'$ in $\p(\h)$: $Q'=e^{iX}Qe^{-iX}$ for $X^*=X$ $Q$-codiagonal with $\|X\|<\pi/2$. 
As remarked in \cite{cpr}, the fact that $X$ is $Q$-codiagonal implies that $X$ anti-commutes with $2Q-1$. Then
$$
2Q'-1=e^{iX}(2Q-1)e^{-iX}=e^{2iX}(2Q-1).
$$
Thus
$$
e^{2iX}=(2Q'-1)(2Q-1).
$$
Since $\|2iX\|<\pi$, the spectrum of $(2Q'-1)(2Q-1)$ is contained in the subset $\{e^{it}: t\in(-\pi,\pi)\}$ of the unit circle, and thus $X$ can be recovered as a continuous (in fact holomorphic) logarithm of $(2Q'-1)(2Q-1)$,
$$
X=-\frac{i}{2} \log\left((2Q'-1)(2Q-1)\right).
$$
Note that  both $(2Q'-1)(2Q-1)P_0$ and $P_0(2Q'-1)(2Q-1)$ are of the form $P_0$ plus compact. It follows that $[(2Q'-1)(2Q-1),P_0]$ is compact, and thus $(2Q'-1)(2Q-1)\in\u_{res}(P_0)$. This 
implies that the exponent $X$ belongs to $\a_{P_0}$ (recall that the exponential map is a diffeomorphism between 
 exponents $X^*=X$  in $\a_{P_0}$ of norm less than $\pi$ and unitaries $U$ in $\u_{res}(P_0)$ such that $\|U-1\|<2$). 
\end{proof}
\begin{rem}\label{remark geometria}
In particular, the above result  provides a way to parametrize elements $Q'\in\c(P_0)$ in the vicinity of a given $Q\in\c(P_0)$. Namely, let 
$$
\v_Q=\{Q'\in\c(P_0): \|Q'-Q\|<1\}.
$$
For each $Q'\in\v_Q$, there exists a unique $X=X_Q(Q')$, $X^*$, $\|X\|<\pi/2$, which is $Q$-codiagonal and belongs 
to $\a_{P_0}$, such that $e^{iX}Qe^{-iX}=Q'$. 

Conversely, to each $X$ as above, there corresponds an element $Q'=e^{iX}Qe^{-iX}\in\c(P_0)$, with $\|Q'-Q\|<1$. Both maps
$$
Q'\mapsto X \ \hbox{ and } \ X\mapsto Q'
$$
are smooth, and each one is the inverse of the other. Thus one has defined a local chart $\v_Q$ for any 
$Q\in\c(P_0)$, which is modelled in an open ball of $(\a_{P_0})_h$.
\end{rem}
\begin{coro}
For any $P_0\in\p_\infty(\h)$, the set $\c(P_0)$ is a smooth manifold modelled in $\a_{P_0}(\h)_h$.
\end{coro}
As remarked in Section 2, the subset $\c_0$ of pairs in $\c$ where either of the projections have finite rank, decomposes as a discerete union of components parametrized by rank and nullity, It is nor difficult to prove that each one of these components are differentiable manifolds. We are though intereseted in the non trivial pairs in $\c$: $\c\setminus \c_0$, comprising the components of $\c_1$ and $\c_\infty$.

\begin{teo}
The set 
$$
\c'=\c\setminus \c_0=\{(P,Q): P,Q \in\p_\infty(\h), PQ \hbox{ is compact}\}
$$
is a smooth differentiable manifold.
\end{teo}
\begin{proof}
Fix a pair $(P_0,Q_0)\in\c'$. We shall exhibit a local chart for $\c'$ near this pair. Let $(P,Q)\in\c'$ such that $\|P-P_0\|<1$. Then, as remarked above, there exists $X=X(P)$ (a smooth map in terms of $P$, with $X(P_0)=0$), $X^*=X$, $\|X\|<\pi/2$ and $X$ is $P_0$-codiagonal, such that 
$$
P=e^{iX}P_0 e^{-iX}.
$$
Then the pair $e^{-iX}(P,Q)e^{iX}=(P_0,e^{-iX}Qe^{iX})$ belongs to $\c(P_0)$. Let $(P,Q)$ be close enough to $(P_0,Q_0)$ so that $e^{-iX}Qe^{iX}$ lies in the local chart $\v_{Q_0}$ for $\c(P_0)$ around $Q_0$  constructed above. Note that if $P\to P_0$, then $e^{iX}\to 1$, so that
$$
\|e^{-iX}Qe^{iX}-Q_0\|\le \|e^{-iX}Qe^{iX}-Q\|+\|Q-Q_0\|
$$
is arbitrarily small if $(P,Q)$ is close  to $(P_0,Q_0)$. The chart for $(P_0,Q_0)$ is  the open set
$$
\v_{(P_0,Q_0)}=\{(P,Q)\in\c': \|P-P_0\|<1 \hbox{ and } e^{-iX}Qe^{iX}\in\v_{Q_0}\}.
$$
If $e^{-iX}Qe^{iX}\in\v_{Q_0}$, then there exists a unique  $Y=X_{Q_0}(e^{-iX}Qe^{iX})$ in $\a_{P_0}$, $Y^*=Y$, $\|Y\|<\pi/2$, which is $Q_0$-codiagonal, such that 
$$
e^{-iX}Qe^{iX}=e^{iY}Q_0e^{-iY}.
$$
Denote $\b_{P_0}=\{X\in\b_h(\h):\|X\|<\pi/2 \hbox{ and } X \hbox{ is } P_0-\hbox{codiagonal}\}$ (and accordingly consider $\b_{Q_0}$). Consider the map 
$$
\Psi=\Psi_{(P_0,Q_0)}:\v_{(P_0,Q_0)}\to \b_{P_0}\times (\b_{Q_0}\cap(\a_{P_0})_h)\subset \b_h(\h)\times (\a_{P_0})_h
$$
given by
$$
\Psi(P,Q)=(X,Y).
$$
The inverse of $\Psi$ is the map
$$
\Psi^{-1}(X,Y)=(e^{iX}P_0e^{-iX},e^{iY}e^{iX}Q_0e^{-iY}e^{-iX}).
$$
\end{proof}

Let us return to $\c(P_0)$ for a fixed $P_0\in\c$, and the fact stated in Remark (\ref{remark geometria}). This remark says a bit more about the geometry of $\c(P_0)$ as a submanifold of $\p(\h)$. Recall from the facts pointed out at the beginning of this section, that two projections at distance less than one are joined by a unique minimal geodesic.
\begin{coro}\label{<1}
Let $Q,Q'\in\c(P_0)$ such that $\|Q-Q'\|<1$. Then the unique geodesic of $\p(\h)$ remains inside $\c(P_0)$.
\end{coro}
\begin{proof}
If $Q,Q'\in\c(P_0)$ with $\|Q-Q'\|<1$, then then the unique (selfadjoint, $Q$-codiagonal) exponent 
$X=X_Q(Q')$ with $\|X\|<\pi/2$ such that $e^{iX}Qe^{-iX}=Q'$, belongs to $\a_{P_0}$.
\end{proof}

Recall from Remark \ref{geometria de P}, the fact that if a weaker condition holds, namely
$$
R(Q)\cap N(Q')=N(Q)\cap R(Q')=\{0\},
$$
then there exists a unique $X$ as above. A natural question is the following. Suppose that this condition for 
uniqueness holds, but $\|Q-Q'\|=1$,  does the unique geodesic $\gamma(t)=e^{iX}Qe^{-iX}$ lie  in $\c(P)$?

{\sc (Esteban Andruchow)} {Instituto de Ciencias,  Universidad Nacional de Gral. Sar\-miento,
J.M. Gutierrez 1150,  (1613) Los Polvorines, Argentina and Instituto Argentino de Matem\'atica, `Alberto P. Calder\'on', CONICET, Saavedra 15 3er. piso,
(1083) Buenos Aires, Argentina. e-mail: eandruch@ungs.edu.ar}

{\sc (Gustavo Corach)} {Instituto Argentino de Matem\'atica, `Alberto P. Calder\'on', CONICET, Saavedra 15 3er. piso, (1083) Buenos Aires, Argentina, and Depto. de Matem\'atica, Facultad de Ingenier\'\i a, Universidad de Buenos Aires, Argentina. e-mail: gcorach@fi.uba.ar}

\end{document}